\documentclass[12pt]{amsart}
\usepackage{amssymb,amsmath,amscd,graphicx,
latexsym,amsthm}
\usepackage{amssymb,latexsym,amsmath,amscd,graphicx,array}
\usepackage[mathscr]{eucal}
  \usepackage[all]{xy}
\usepackage{hyperref}

\usepackage[all]{xy} \renewcommand{\P}{ \ensuremath{\mathbb{P}}}
\usepackage{layout}

\title{Volume functions of linear series}

\theoremstyle{plain}
\newtheorem{theorem}{Theorem}[section]
\newtheorem*{theoremA}{Theorem A}
\newtheorem*{theoremB}{Theorem B}
\newtheorem{lemma}[theorem]{Lemma}
\newtheorem{proposition}[theorem]{Proposition}

\newtheorem*{qsn}{Question}

\theoremstyle{definition}

\newtheorem{remark}[theorem]{Remark}

\def\Z{{\mathbb Z}} 
\def\N{{\mathbb N}} 
 
\def\C{{\mathbb C}} 
 
\def\R{{\mathbb R}} 
\def\Q{{\mathbb Q}} 
\def\P{{\mathbb P}} 
\def\OO{{\mathcal O}} 
\newcommand{\HH}[3]{H^{{#1}} \big( {#2} , {#3} 
\big) } 
\newcommand{\hh}[3]{h^{{#1}} \big( {#2} , {#3} 
\big) }

\newcommand{\deq}{\ensuremath{ \stackrel{\textrm{def}}{=}}} 
\newcommand{\equ}{\ensuremath{ \,=\, }}

\newcommand{\dleq}{\ensuremath{ \,\leq\, }}

\newcommand{\Amp}{\textup{Amp}(X)_{\mathbb{R}}}

\newcommand{\Nef}{\textup{Nef}(X)_{\mathbb{R}}}
\newcommand{\Neft}{\textup{Nef}(X_t)_{\mathbb{R}}}
\newcommand{\bull}{\bullet}

\newcommand{\st}[1]{\ensuremath{ \left\{ #1 \right\} }}

\DeclareMathOperator{\NN}{N}
\DeclareMathOperator{\vol}{vol}

\begin{document}

\author{Alex K\"uronya}
\address{Budapest University of Technology and Economics, Department of Algebra, Budapest P.O. Box 91, H-1521 Hungary}
\email{{\tt alex.kuronya@math.bme.hu}}

\author{Victor Lozovanu}
\address{University of Michigan, Department of Mathematics, Ann Arbor, MI 48109-1109, USA}
\email{{\tt vicloz@umich.edu}}

\author{Catriona Maclean}
\address{Universit\'e Joseph Fourier, UFR de Math\'ematiques, 100 rue des Maths, BP 74, 38402 St Martin d'H\'eres, France}
\email{\tt Catriona.Maclean@ujf-grenoble.fr}

\begin{abstract}
The volume of a Cartier divisor is an asymptotic invariant, which measures the rate of growth of sections of powers of the divisor. 
It extends to a continuous, homogeneous, and log-concave function on the whole N\'eron--Severi space, thus giving rise to a basic invariant of the underlying
projective variety. Analogously, one can also define the volume function of a possibly non-complete multigraded linear series. 

In this paper we will address the
question of characterizing the class of functions arising on the one hand as  volume functions of multigraded linear series and on the other hand as  volume
functions of projective varieties. 

In the multigraded setting, inspired by the work of Lazarsfeld and Musta\c t\u a \cite{LM08} on Okounkov bodies, we show that
any continuous, homogeneous, and log-concave function appears as the volume function of a multigraded linear series. By contrast we show that there exists
countably many functions which arise as the volume functions of projective varieties. We end the paper with an example, where the volume function of a
projective variety is given by a transcendental formula, emphasizing the complicated nature of the volume in the classical case.
\end{abstract}

\thanks{During this project the first author was partially supported by  the CNRS, the  DFG-Leibniz program, the SFB/TR 45 ``Periods, moduli spaces and arithmetic of algebraic varieties'',  the OTKA Grants 61116,77476, 77604, and a Bolyai Fellowship by the Hungarian Academy of Sciences. }

\maketitle
\section*{Introduction}
Let $X$ be a smooth complex projective variety of dimension $n$ over the complex numbers,  let $D$ be a Cartier divisor on $X$. The volume of $D$ is defined as 
\[
\textup{vol}_{X}(D) \ = \ \limsup_{k\rightarrow \infty}\frac{\textup{dim}_{\mathbb{C}}(H^0(X,\mathcal{O}_X(kD)))}{k^n/n!}\ .
\]
The volume and its various versions have recently played  a crucial role in several important developments in higher dimensional geometry, see for example
\cite{Ta06}, \cite{HM06}. 

In the classical setting of ample divisors, the volume of $D$ is simply its top self-intersection. Starting with the work of Fujita \cite{Fu94}, Nakayama \cite{Na04}, and Tsuji \cite{Ts99}, it became gradually clear that the volume of big divisors --- that is, ones with  $\vol_X(D)>0$ --- displays a surprising number of properties analogous to that of ample ones. Notably,  it depends only on the numerical class of $D$, it is homogeneous of degree $n$, and satisfies a Lipschitz-type property (\cite[Section 2.2.C]{Laz04}). Consequently, one can extend the volume to a continuous function 
\[
\textup{vol}_X \ : \ N^1(X)_{\R} \ \longrightarrow \ \R_+\ ,
\]
where by $N^1(X)_{\R}$ we mean the finite-dimensional real vector space of numerical equivalence classes of $\R$-divisors. Besides continuity and
homogeneity, another important feature of the volume function is log-concavity of degree $n$, i.e. for any two classes $\xi , \xi'\in \textup{Big}(X)_{\R}$
we have 
\[
\textup{vol}_X(\xi +\xi')^{1/n} \ \geq \ \textup{vol}_X(\xi )^{1/n} \ + \ \textup{vol}_X(\xi')^{1/n}\ .
\]
Given a sufficient amount of  information, the volume function associated to a variety can be explicitly computed under certain circumstances. Examples include
all  smooth surfaces \cite{BKS}, toric varieties \cite{HKP}, and homogeneous spaces. 

However, beside what has been mentioned above,  relatively little is known about its global behavior, and understanding it more clearly remains a very important
quest.\footnote{In their interesting paper \cite{BFJ07}, Boucksom-Favre-Jonsson found a nice formula for the derivative of $\textup{vol}_X$ in any direction.}

In \cite{LM08}, Lazarsfeld and Musta\c t\u a showed that in fact most of the properties of $\textup{vol}_X$ are quite formal in nature, and their validity can be extended to the non-complete multigraded setting. Specifically, fix a choice of Cartier divisors $\mathbf{D}=(D_1,\ldots ,D_{\rho})$ on $X$ (where $\rho$ is an arbitrary positive integer for the time being, but soon it will be $\dim_\R \NN^1(X)_\R$), and set $\mathbf{m}\mathbf{D}=m_1D_{1}+\ldots +m_{\rho}D_{\rho}$ for any $\mathbf{m}=(m_1,\ldots ,m_{\rho})\in \N^{\rho}$. A multigraded linear series $W_{\bullet}$ on $X$ associated to $D_1,\ldots ,D_{\rho}$ consists of subspaces
\[
W_{\mathbf{m}} \ \subseteq \ H^0(X,\mathcal{O}_X(\mathbf{m}\mathbf{D}))\ ,
\]
such that $R(W_{\bullet}) \equ \oplus W_{\mathbf{m}}$ 
is a subalgebra of the section ring 
\[
R(D_1,\ldots ,D_{\rho})=\oplus_{\mathbf{m}\in\N^{\rho}} H^0(X,\mathcal{O}_X(\mathbf{m}\mathbf{D}))\ .
\] 
The support of $W_{\bullet}$ is then defined to be the  closed convex cone in  $\R^{\rho}_+$ spanned by all multi-indices $\mathbf{m}\in\N^{\rho}$ such that
$W_{\mathbf{m}}\neq 0$. Given $\mathbf{a}\in\N^{\rho}$, set 
\[
\textup{vol}_{W_{\bullet}}(\mathbf{a}) \deq   
\limsup_{k \to \infty} \frac{\textup{dim}_{\C}(W_{k\cdot\mathbf{a}})}{k^n/n!}\ .
\]
Exactly as in the complete case, the above assignment defines the volume function of $W_{\bullet}$
\[
 \textup{vol}_{W_{\bullet}} \ : \ \N^{\rho} \ \longrightarrow \ \R_+\ .
\]
Based on earlier  work of Okounkov, \cite{Ok96} and \cite{Ok03}, the authors of \cite{LM08} associate a convex cone --- the  so-called  Okounkov cone --- to a
multigraded linear series on a projective variety. With the help of  convex geometry and semigroup theory they show that  the formal
properties of the global volume function persist in the multigraded setting under very mild hypotheses. 

Precisely as in the global case, the function
$\mathbf{m}\mapsto \textup{vol}_{W_{\bullet}}(\mathbf{m})$ extends uniquely to a continuous function
\[
\textup{vol}_{W_{\bullet}} \ : \ \textup{int(supp(}W_{\bullet})) \ \longrightarrow \ \R_+\ ,
\]
which is homogeneous,  log-concave of degree $n$, and extends continuously to the entire $\textup{supp}(W_{\bullet})$. The  construction generalizes the
classical case: whenever $X$ is an irreducible  projective variety, the cone of big divisors $\textup{Big}(X)_{\R}$ is pointed and $\textup{vol}_X$ vanishes
outside of it. Pick Cartier divisors  $D_1,\ldots ,D_{\rho}$ on $X$, whose classes in $N^1(X)_{\R}$ generate a cone containing $\textup{Big}(X)_{\R}$. Then
$\textup{vol}_X=\textup{vol}_{W_{\bullet}}$ on $\textup{Big}(X)_{\R}$, where $W_{\bullet}=R(D_1,\ldots ,D_{\rho})$.

In this "in vitro"  setting, we prove first that in fact any continuous, homogeneous, and log-concave function arises as the volume function of an appropriate multigraded linear series.

\begin{theoremA}\label{Theorem A}
Let $K\subseteq \R^{\rho}_+$ be a closed convex cone with non-empty interior, $f:K\rightarrow \R_+$  a continuous function, which is non-zero, homogeneous, and
log-concave of degree $n$  in the interior of $K$. Let $X$ be an arbitrary smooth, projective toric variety of dimension $n$ and Picard number $\rho$. Then there
exists  a multigraded linear series $W_{\bullet}$ on $X$ such that $\textup{vol}_{W_{\bullet}} \equiv f$ on the interior of
$K$. Moreover we have $\textup{supp}(W_{\bullet})=K$.
\end{theoremA}

As a consequence, notice that the volume function $\textup{vol}_{W_{\bullet}}$ of a multigraded linear series $W_{\bullet}$ can be pretty wild. This is due to the following connection. Alexandroff~\cite{Al39} showed that a function as in Theorem A is almost everywhere twice differentiable; at the same time, one can give examples of functions of this sort, which are nowhere three times differentiable (cf. Remark~\ref{nowhere differentiable}). This gives a positive answer to \cite[Problem 7.2]{LM08}. 

For the proof of Theorem A we first check that any function as in the statement is the Euclidean volume function of a pointed cone. Then using toric geometry we associate to a multigraded linear series a pointed cone. We finish the proof of Theorem A by giving a recipe for the inverse process,  constructing a multigraded linear series from a cone. 

It follows from Theorem A that there exist  uncountably many volume functions in the non-complete case. In comparison, in the complete  case we prove that in fact there are only countably many of them:

\begin{theoremB}\label{Theorem B}
Let $V_{\Z} =\Z^{\rho}$ be a lattice inside the vector space $V_{\R}=V_{\Z}\otimes_{\Z}\R$. Then there exist countably many functions $f_j:V_{\R}\rightarrow
\R_+$ with $j\in\N$, so that for any irreducible projective variety $X$ of dimension $n$ and Picard number $\rho$, we can construct an integral linear
isomorphism          
\[
\pi_{X} : \ V_{\R} \ \rightarrow \ N^1(X)_{\R}
\]
with the property that $\textup{vol}_X \ \circ \ \pi_X \ =\ f_j$ for some $j\in\N$.
\end{theoremB}
\noindent
We prove Theorem B in the case of smooth varieties. The general case follows easily by appealing to resolution of singularities. The heart of the proof is a careful analysis of  the variation of the volume function in families coming from multi-graded Hilbert schemes. This approach enables us to establish analogous statements for the ample, nef, big, and pseudoeffective cones. We would like to point out that the countability of ample or nef cones also follows from the work of Campana and Peternell \cite{CP90} on the algebraicity of these cones. 

An amusing application of Theorem B concerns the set of volumes $\mathbb{V}\subseteq\R_+$, which is the set of all non-negative real numbers arising
 as the volume of a Cartier divisor on some irreducible projective variety. Using Theorem B, one can deduce that $\mathbb{V}$ has the structure of a countable
multiplicative semigroup (cf.~Remark \ref{semigroup of volumes}). By contrast, in the last section we give an example of a four-fold whose volume function is
given by a transcendental function, deepening further the mystery surrounding the volume function in the classical case. In particular, the same example
provides a Cartier divisor with transcendental volume, thus the set of volumes $\mathbb{V}$ contains  transcendental numbers as well.

\subsection*{Acknowledgments.} Part of this work was done while the first and the second authors were enjoying the hospitality of the Universit\'e Joseph
Fourier in Grenoble. We would like to use this opportunity to thank Michel Brion and the Department of Mathematics for the invitation. We are grateful to
Sebastien Boucksom, Rob Lazarsfeld, and Mircea Musta\c t\u a for many helpful discussions. Special thanks are due to an anonymous referee for suggestions
leading to significant expository improvements and notable strengthening of
the results of section 1.


\section{Volume functions of non-complete linear series}

In this section we study the volume function of a multigraded linear series and  verify that
any non-zero continuous homogeneous log-concave function appears
 as the volume function of some multigraded linear series on a smooth 
projective toric variety of dimension $n$.

First we  introduce some notation. In the course of this section $X$ will be a smooth projective toric variety of dimension $n$ with an action of a torus $T=N\otimes_{\Z}\C^*$. Denote by $M\simeq\Z^n$ the character lattice of $T$ and assume the existence of a fixed isomorphism $\alpha: M\rightarrow \Z^n$. In particular, the Euclidean volume of an open set in $M_\R$ is well-defined.

For any choice $\mathbf{D}=(D_1,\dots,D_\rho)$ of (not necessarily different) 
toric Cartier divisors on X define
\[
 \Gamma(\mathbf{D}) \deq\st{(u,\mathbf{m})\in M\times\N^\rho\mid 
\chi^u\in\HH{0}{X}{\OO_X(\mathbf{m}\mathbf{D})}}\ .
\]
Note that $\Gamma(\mathbf{D})$ is finitely generated, and spans a 
polyhedral cone 
$\Delta(\mathbf{D})$ in $M_\R\times\R_+^\rho$. Note moreover that if $k$
is a natural number then
\[ \Delta(k\mathbf{D})=\{ (kv, w)| (v,w)\in\Delta(\mathbf{D})\}.\]

\begin{theorem} \label{noncomplete volume function}
Let $K\subseteq \R^\rho_+$ 
be a closed convex cone with non-empty interior, $f:K\rightarrow \R_+$ 
be a continuous function, which is non-zero, homogeneous and
log-concave of degree $n$ in the interior of $K$. Then there exists a 
multi-divisor $\mathbf{D}=(D_1,\ldots,D_\rho)$, and  a multigraded linear series 
$W_{\bullet}$ on $X$ with $
W_{\mathbf{m}} \ \subseteq \ H^0(X,\mathcal{O}_X(\mathbf{m}\mathbf{D})))$
for any $\mathbf{m}\in\N^{\rho}$, such that $\textup{supp}(W_{\bullet})=K$ and 
$\textup{vol}_{W_{\bullet}}= f$ on the interior of $K$.
\end{theorem} 

\begin{remark}\label{nowhere differentiable}
It is relatively easy to use Theorem~\ref{noncomplete volume function} to produce examples of linear series whose volume function is not very regular. Indeed, functions as in Theorem~\ref{noncomplete volume function} can be constructed from continuous concave functions $g:B\rightarrow \R_+$ defined on a bounded convex body  $B\subseteq \R^{\rho -1}$. For this let $H\subseteq \R_+^{\rho}$ be an affine hyperplane, not containing the origin, such that $H\cap K=B$ 
is bounded. The function
\[
g \deq  \sqrt[n]{f} \ : \ B \rightarrow \R_+
\]
then extends uniquely to a function on $K$ satisfying the hypotheses of 
Theorem~\ref{noncomplete volume function}. In dimension one for example, continuous concave functions can be generated from arbitrary negative bounded continuous functions by integrating twice. So taking a negative, bounded and nowhere differentiable continuous function defined on a closed interval, and integrating it twice,  we obtain a continuous, concave and nowhere three 
times differentiable function.
\end{remark}


\begin{proof}[Proof of Theorem~\ref{noncomplete volume function}]
We consider a toric variety $X$, together with a multi-divisor 
$\mathbf{D}=(D_1,\ldots, D_\rho)$. 
Throughout what follows, we will say that a cone $C\subseteq M_\R\times \R^\rho_+$ is   \emph{$M_\R$-bounded}, if 
 for any $p\in \R^\rho_+$ the volume of the slice 
$\{v\in M_\R \ |\  (v,p)\in C\}$ is finite.

For any $M_\R$-bounded cone $C$ let ${\rm vol}_C: \R_+^\rho\rightarrow \R$ be the function given by the formula
\[
{\rm vol}_C(p) \ = \ {\rm vol}(\{v\in M_\R| (v,p)\in C\}), \text{ for any } p\in\R^{\rho}_+
\]
In the next lemma we reduce to the case where $f=n!{\rm vol}_C$
for some closed convex cone $C$. 
\begin{lemma}\label{lemma:function}
If $f:K \longrightarrow \R_+$ is a function as in Theorem A, then 
there exists an $M_\R$-bounded closed convex 
cone in $C\subseteq M_\R\times K$ such
that 
$${\rm vol}_C(v) \ = \ f(v)/n! $$ for all $v\in \textup{int}(K)$.
\end{lemma}
\begin{proof}
Using the isomorphism $\alpha$, it will be enough to construct $C$ in
$\R^n\times K$. The cone
\[
 C\deq \st{(v,p)\in\R^n_+\times K\mid \sum_i{v_i} \leq (f(p))^{1/n}}
\]
has the required properties. 
\end{proof}
From now on, we assume given an $M_ \R$-bounded 
cone $C\subset  M_\R\times K$ such that 
${\rm vol}_C= f$ on ${\rm int}(K)$.  
The following lemma proves the theorem in the case where 
$C\subseteq \Delta(\mathbf{D})$.
\begin{lemma}\label{subcone}
Let $X$, $\mathbf{D}$ and $C$ be as above, and
assume that the multi-divisor $\mathbf{D}$ is such that 
$C\subseteq \Delta(\mathbf{D})$. 
There is then a multigraded linear series $W_\bullet$ with
$W_{\mathbf{m}}\subset H^0(\mathcal{O}_X(\mathbf{m}\mathbf{D}))$ such that 
${\rm vol}(W_{\bullet})= n!{\rm vol}_C $ on ${\rm int}(K)$. 
\end{lemma}
\begin{proof}
We define $W_{\bullet}$ as follows. 
For any $\mathbf{m}\in\N^\rho$ we set 
\[
 W_{\mathbf{m}}\deq\langle \chi^u\in \HH{0}{X}{\OO_X(\mathbf{m}\mathbf{D})}\mid u\in C\cap (M\times\st{\mathbf{m}}) \rangle \ .
\]
By construction $W_{\mathbf{m}} \cdot W_{\mathbf{n}} \subseteq W_{\mathbf{m+n}}$ for all $\mathbf{m},\mathbf{n}\in\N^\rho$ and 
\[
 \dim W_{\mathbf{m}} \equ \# \big(C\cap (M\times\st{\mathbf{m}})\big)\ ,
\]
hence $\vol_{W_\bullet}(\mathbf{m})\equ n!\vol_{C}(\mathbf{m})$. This completes the proof of Lemma \ref{subcone}.
\end{proof}
Note in particular that if $C$ is the cone whose existence is guaranteed by 
Lemma \ref{lemma:function} then we have that $\vol_{W_\bullet}= f$ on 
${\rm int}(K)$. 
The function 
\[
\textup{vol}_{W_{\bullet}}:\textup{int(supp}(W_{\bullet})) \rightarrow \R_+
\]
is proportional to the $\textup{vol}_C$. It is therefore continuous, homogeneous and log-concave of degree $n$ in the interior of $\textup{supp}(W_{\bullet})$. Since $W_{\bullet}$ is a subseries of the complete multigraded linear series defined by $\mathbf{D}$,  the function $(\textup{vol}_{W_{\bullet}})^{1/n}$ is bounded in the sense that
\[
\big(\textup{vol}_{W_{\bullet}}(v)\big)^{1/n} \ \leq \ k_1||v||, \textup{ for all } v\in \textup{supp}(W_{\bullet})
\]
for some $k_1>0$. The concavity of the function
$(\textup{vol}_{W_{\bullet}})^{1/n}$, implies 
that it satisfies a H\"older
condition of exponent  $1$ (see  \cite[Theorem 1.5.1]{Sch93})
\[
|\big(\textup{vol}_{W_{\bullet}}(v)\big)^{1/n}-\big(\textup{vol}_{W_{\bullet}}(w)\big)^{1/n}| \ \leq \ k_2||v-w||
\]
for all $v,w\in \textup{int(supp}(W_{\bullet}))$. These conditions imply that the
function $\textup{vol}_{W_{\bullet}}$ 
can be extended continuously to the whole support of $W_{\bullet}$.
\noindent
To complete the proof of the theorem, it will be enough to establish the 
following lemma.
\begin{lemma}\label{scaling}
Let $C$ be an $M_\R$-bounded closed convex  
subcone of $M_\R\times K$ and let $\mathbf{D}=(D_1,\ldots, D_\rho)$ be 
a choice of big Cartier divisors on $X$. There is then
a linear transformation $\phi: M_\R\times \R^\rho\rightarrow M_\R\times \R^\rho$ 
and an integer $k$ such that
\begin{enumerate}
\item $\phi(b,0) = (b, 0)$ and $\phi(b,a)= (*,a)$ for all $b\in M_\R$ and $a\in\R^\rho$.
\item $\phi(C)\subset \Delta(k\mathbf{D})$.
\end{enumerate}
\end{lemma}
\begin{proof}
For every $i$, we consider the set $B_i=\{w\in M_\R| (w, v_i)\in \Delta(\mathbf{D})\}$, where $v_i$ is the $i$th unit vector. The fact that $D_i$ is big implies that $B_i$ has non-empty interior. For each $i$ we pick an element $d_i$ in the interior of $B_i$ and we consider the map $\psi: \R^\rho\rightarrow M_\R$ given by
\[
\psi(a_1,\ldots, a_\rho)= \sum a_i d_i.
\]
This map has the property that for any $\underline{a}\in\R_+^\rho$, $(\psi(\underline{a}),\underline{a})$ is contained in the interiorof  $\Delta(\mathbf{D})$. There is a $\delta>0$ such that for any $\underline{a}\in \R^\rho_+$ and any $\underline{b}\in M_\R$ we have that 
\[
||\underline{b}||< \delta ||\underline{a}||\Rightarrow (\psi(\underline{a})+\underline{b}, \underline{a})\in\Delta(\mathbf{D}).
\] 
The cone $C$ being $M_\R$-bounded, there is a  $l$ such that $(v,w)\in C\Rightarrow ||v||\leq l || w ||$. We choose an integer $k>l/\delta$ and consider the map 
\[ 
\phi(v,w) = (v +k\psi(w),w).
\]
Suppose that $(v,w)\in C$. We then have that  
\[ 
||v/k|| < || \delta v/l || < \delta || w|
|\]
which implies that 
\[ 
(v/k+\psi(w), w)\in \Delta(\mathbf{D})
\]
and hence
\[
 \phi(v,w)= (v+k\psi(w), w)\in \Delta(k\mathbf{D}).
 \]
This completes the proof of Lemma \ref{scaling}. 
\end{proof}
The proof of Theorem 1.1 is now complete. Given a function $f$ satisfying the 
hypotheses of Theorem 1.1, we can find an $M_\R$-bounded closed convex cone $C$
such that ${\rm vol}_C=f$. After applying Lemma \ref{scaling} (and replacing
$C$ by $\phi(C)$ if necessary) we can assume there is a multi-divisor 
$\mathbf{D}$ such that $C\subset \Delta(\mathbf{D})$. By Lemma 
\ref{subcone} we obtain a linear series $W_\bullet$ such that 
\[{\rm vol}(W_\bullet)= n!{\rm vol}_C=f\]
on ${\rm Int}(K)$. 
\end{proof}

\begin{remark}
Interestingly enough, Theorem~\ref{noncomplete volume function} proves significant already  in the simplest meaningful case, when $X=\P^1$.  More specifically,
let $f:\R^r_+\to\R_+$ be a continuous, concave, and $1$-homogeneous function. After possibly scaling $f$, we can assume that $f(x)\leq \sum_{i}x_i$. Now, for each $\mathbf{m}\in\N^r$, define 
\[
 W_\mathbf{m} \deq \langle \chi^u\mid 0\leq u\leq f(\mathbf{m})\rangle \subseteq \HH{0}{\P^1}{\mathcal{O}(|\mathbf{m}|)}\ .
\]
Then $\dim (W_{\mathbf{m}}) \equ f(\mathbf{m})+1$  and we have $\vol_{W_\bullet} = f$ on $\R^r_+$. 
\end{remark}

\section{Countability of volume functions for complete linear series.}

One of the consequences of the previous section is that for non-complete multigraded linear series there are uncountably many different volume functions. By contrast, we will prove that there are only countably many volume functions for all irreducible projective varieties. 

\begin{theorem}\label{Theorem}
Let $V_{\Z} =\Z^{\rho}$ be a lattice inside the vector space $V_{\R}$. Then 
there exist countably many closed convex cones $A_i\ \subseteq \ V_{\R}$ and functions 
$f_j:V_{\R}\rightarrow \R$ with $i, j\in\N$, so that for any smooth 
projective variety $X$ of dimension $n$ and Picard number $\rho$, 
we can construct an integral linear isomorphism
\[
\rho_{X} : \ V_{\R} \ \rightarrow \ N^1(X)_{\R}
\]
with the properties that
\[
\pi^{-1}_{X}(\Nef) \ = \ A_i, \text{  and   } \textup{vol}_X \ \circ \ \pi_X \ = 
\ f_j
\]
for some $i,j\in\N$.
\end{theorem}

\begin{remark}\label{irreducible varieties}
$(1)$ Theorem~\ref{Theorem} quickly implies Theorem B: let $X$ be an irreducible projective variety  and let $\mu :X'\rightarrow X$ be a resolution of singularities of $X$. The pullback map
\[
\mu^* \ : \ N^1(X)_{\R} \rightarrow N^1(X')_{\R}
\]
is linear, injective,  and $\textup{vol}_{X}=\textup{vol}_{X'}\circ \mu^*$ by \cite[Example 2.2.49]{Laz04}. Since the map $\mu^*$ is defined by choosing dim$(N^1(X)_{\R})$ integral vectors, the countability of the volume functions in the smooth case implies that the same statement is valid for the collection of irreducible varieties. As  $\textup{Nef}(X)_{\R}=(\mu^*)^{-1}(\textup{Nef}(X')_{\R})$, the same statement holds for nef  cones.\\
$(2)$ Since $\Amp=\textup{int}(\Nef) $, then Theorem B remains valid for  ample cones as well. Much the same way,  the  cone of big divisors can be described as 
\[
\textup{Big}(X)_{\R}\ = \ \{ D\in N^1(X)_{\R} \ | \ \textup{vol}_X(D) 
\ > \ 0 \ \}\ ;
\]
its  closure is known to be equal to the pseudo-effective cone $\overline{\textup{Eff}(X)}_{\R}$. Hence we conclude that Theorem B is also valid for the big and pseudoeffective cones.
\end{remark}
\begin{remark}[The semigroup of volumes]\label{semigroup of volumes}
Let
\[
\mathbb{V} \deq \{ a\in\R_+ \ | \ a=\textup{vol}_X(D) \textup{ for some pair } (X,D)\}
\] 
where $X$ is some  irreducible projective variety and $D$ a Cartier divisor on $X$. By Theorem~\ref{Theorem},  $\mathbb{V}$ is countable. Moreover, using the K\"unneth formula \cite[Proposition 4.5]{BKS}, one can show that the set $\mathbb{V}$ has the  structure of a multiplicative semigroup with respect to the product of real numbers. Beyond this fact very little is known about $\mathbb{V}$. It is certainly true by \cite[Example 2.3.6]{Laz04} and the semigroup structure of  $\mathbb{V}$  that all non-negative  rational numbers are contained in $\mathbb{V}$. At the same time we do not know whether all algebraic numbers appear as volumes of Cartier divisors. Going in the other direction,  we provide  an example in Section $3$ of a pair $(X,D)$ such that the volume $\textup{vol}_X(D)$ is transcendental.
\end{remark}

We will make  some preparations. Let $\phi:\mathcal{X}\rightarrow T$ be a smooth projective and surjective morphism of relative dimension $n$ between two quasi-projective varieties. Suppose further that $T$ and each fiber of $\phi$ is irreducible and reduced. If we are given $\rho$ Cartier divisors $D_1,\ldots ,D_{\rho}$ on $\mathcal{X}$ then we say that a closed point $t_0\in T$ admits a \emph{good fiber} if $D_1|_{X_{t_0}},\ldots ,D_{\rho}|_{X_{t_0}}$ form a basis for the N\'eron--Severi space $\text{N}^1(X_{t_0})_{\R}$. The main ingredient of  the proof of Theorem~\ref{Theorem} is the following statement.

\begin{proposition}\label{linear independent}
Let $\phi:\mathcal{X}\rightarrow T$ be a family as above and suppose that there exists a closed point $t_0\in T$, admitting a good fiber. Then for all closed points $t\in T$ the Cartier divisors $D_1|_{X_t},\ldots ,D_{\rho}|_{X_t}$ are linear independent in $N^1(X_t)_{\R}$.
\end{proposition}

\begin{proof}
First notice that $D_1|_{X_t}, \ldots , D_{\rho}|_{X_t}$ are linearly dependent in $N^1(X_t)_{\R}$ if and only if they are linear dependent over integers. Therefore we only need to show that given a Cartier divisor $D$ on $\mathcal{X}$ such that $D|_{X_{t_0}}\neq_{\text{num}}0$, one has  $D|_{X_{t}}\neq_{\text{num}}0$ for any $t\in T$. 

We use induction on the dimension of the fibers. First assume that $\text{dim}(\mathcal{X})=\text{dim}(T)+1$. As $X_{t_0}$ is a smooth irreducible curve, the condition $D|_{X_{t_0}}\neq_{\text{num}}0$ is equivalent to $(D.X_{t_0})\neq 0$. The morphism $\phi$ is smooth and $T$ irreducible, therefore the function
\[
t\in T \longrightarrow (D.X_t)
\]
is globally constant. Consequently, $D|_{X_t}\neq 0$ for any $t\in T$ as we wanted. 

In the general case, when $n\geq 2$, let $t_1\in T\setminus\{ t_0\}$ and choose a line bundle $A$ on $\mathcal{X}$ which is very ample relative to the map $\phi$. Bertini's Theorem and generic smoothness says that for a general section $W$ of $A$, the fiber $W_t=W\cap X_t$ is smooth and irreducible for all $t$'s in some open neighborhood of $t_0$. The same statement holds for $t_1$, and using the fact that $T$ is irreducible, one can choose a general section $W$ and an open neighborhood $U\subseteq T$ containing both $t_0$ and $t_1$, such that $W_t$ is smooth and irreducible for all $t\in U$. Now, as $W$ is general, the map 
\[
\phi_{W}^{U} \ = \ \phi|_{W\cap \phi^{-1}(U)} \ : \ W\cap \phi^{-1}(U) \longrightarrow U
\]
is flat and of relative dimension $n-1$. Because each fiber of $\phi_{W}^{U}$ is smooth, $\phi_W^U$ is smooth as well. With this in hand, suppose that $D|_{W_{t_0}}\neq_{\textup{num}} 0$. By applying induction to the family $\phi_{W}^{U}$, we obtain $D|_{W_{t_1}}\neq_{\textup{num}} 0$, hence  $D|_{X_{t_1}}\neq_{\textup{num}} 0$. 

Whenever $D|_{W_{t_0}}=_{\textup{num}}0$, we have two cases. If $n=2$, we can use the fact that $W_{t_0}$ is an ample section of $X_{t_0}$ and  deduce from the Hodge Index Theorem  that $(D|_{X_{t_0}})^2 <0$. Hence by flatness one obtains that $(D|_{X_{t_1}})^2 <0$ and therefore $D|_{X_{t_1}}\neq_{\text{num}} 0$. When  $n\geq 3$, one can use a higher-dimensional version of the Hodge Index Theorem \cite[Corollary I.4.2]{KL66} and deduce that the condition $D|_{W_{t_0}}=_{\text{num}}0$ implies  $D|_{X_{t_0}}=_{\text{num}}0$, contradicting our assumptions.
\end{proof}

\begin{proof}[Proof of Theorem \ref{Theorem}]
Our first step  is to embed every smooth projective variety $X$ of dimension $n$ and Picard number $\rho$ into a product of projective spaces, i.e.
\[
X \ \subseteq \ Y \ = \ \underbrace{\P^{2n+1} \times \ldots \times \P^{2n+1}}_{\rho \text{ times }}
\]
with the property that the restriction map
\[
\rho_X \ : \ V_{\R} \ := \ N^1(Y)_{\R} \ = \ \R^{\rho} \ \rightarrow \ N^1(X)_{\R}
\]
is an integral linear isomorphism. 

To this end fix  $\rho$ very ample Cartier divisors $D_{1,X},\ldots ,D_{\rho ,X}$ on $X$, which form a $\Q$-base of the N\'eron--Severi group $N^1(X)_{\Q}$. As $X$ is a smooth variety, \cite[Theorem 5.4.9]{Sha94} implies that for each $D_{i,X}$ there exists an embedding $X\ \subseteq \ \P^{2n+1}$ with $\mathcal{O}_X(D_{i,X}) \ = \ \mathcal{O}_{\P^{2n+1}}(1)|_X$. With this in hand, we embed $X$ in $Y$ in the following manner
\begin{equation}\label{embedding}
X \ \subseteq \ \underbrace{X\times \ldots \times X}_{\rho \text{ times }} \ \subseteq Y\ ,
\end{equation}
where the first embedding is given by the diagonal. The corresponding restriction map $\rho_X$ on the N\'eron--Severi groups is an integral linear isomorphism identifying the semigroup $\N^{\rho}\subseteq \text{N}^1(Y)_{\Z}$ with the one generated by  $D_{1,X},\ldots ,D_{\rho ,X}$ in $N^1(X)_{\Z}$.

Next, we construct countably many families  such that  each smooth variety $X$ embedded in $Y$ as in (\ref{embedding}) appears as a fiber in at least one of them. We will use multigraded Hilbert schemes of subvarieties embedded in $Y$ for this purpose. Before introducing them, we  note that each line bundle on $Y$ is of the form
\[
\mathcal{O}_Y(\mathbf{m}) \deq p_{1}^{*}(\mathcal{O}_{\P^{2n+1}}(m_1))
\otimes\ldots \otimes p_{\rho}^{*}(\mathcal{O}_{\P^{2n+1}}(m_{\rho}))
\] 
with $\mathbf{m}=(m_1,\ldots ,m_{\rho})\in \Z^{\rho}$ and $p_i : Y \rightarrow \P^{2n+1}$ being the $i$\textsuperscript{th} projection. For  a closed subscheme $X\subseteq Y$, one can define its multigraded Hilbert function as 
\[
P_{X, Y}(\mathbf{m}) \ = \ \chi(X, (\mathcal{O}_Y(\mathbf{m}))|_X), \text{ for all } \mathbf{m}\in\Z^{\rho}\ .
\]
The Hilbert functor $\mathcal{H}_{Y,P}(T)$ parameterizes families of closed subschemes $Z\subseteq Y\times T$ flat over $T$ such that for any $t\in T$ the multigraded Hilbert function of the scheme-theoretical fiber $Z_t\subseteq Y$ equals $P$. In \cite[Corollary $1.2$]{HS04}, Haiman and Sturmfels prove that this functor is representable, i.e. for any $\rho\geq 1$ and $P$ as above the multigraded Hilbert functor $\mathcal{H}_{Y,P}$ is represented by a projective scheme $\textup{Hilb}_{Y,P}$ and by an universal family
\[
\xymatrix{
U_P \ar[drr]_{\phi} & \subseteq & Y\times \textup{Hilb}_{Y,P} \ar[d]^{pr_2}\\
 &  & \textup{Hilb}_{Y,P}  \\}
\]
with the property that there is a bijection between the closed subschemes of $Y$ with the multigraded Hilbert function equal $P$ and the scheme theoretical fibers of $\phi$. 

In the case when $X$ is a smooth projective variety of dimension $n$ and Picard number $\rho$ embedded in $Y$ as in (\ref{embedding}), its multigraded Hilbert function $P_{X,Y}$ is a polynomial with rational coefficients and of total degree at most $n\cdot\rho$. Hence there are countably many polynomials of this form, and therefore countably many families such that any smooth projective variety of dimension $n$ and Picard number $\rho$ appears as a fiber in at least one of them. 

By what we said above, it is enough to verify  countability for one of these flat families. Fix one of them, and call it $\phi :\mathcal{X}\rightarrow T$. Without loss of generality we can assume that $T$ is irreducible and reduced. The fact that $\phi$ is flat implies  by \cite[Theorem 12.2.4]{Gr} that the set of all $t\in T$ for which $X_t$ is smooth, irreducible, and reduced, is open. 

Arguing inductively on $\dim T$, we can restrict our attention to a non-empty open subset of $T$ and assume that all the fibers of $\phi$ are smooth, irreducible, and reduced. This implies that $\phi$ is smooth, so it is enough to prove  countability under this  additional condition. 

The embedding $\mathcal{X}\subseteq Y\times T$ tells us that $\mathcal{X}$ comes equipped with $\rho$ Cartier divisors $D_1,\ldots ,D_{\rho}$, the restriction of the canonical base of $\textup{Pic}(Y)$. Assume further that there exists a closed point $t_0\in T$ such that $X_{t_0}$ is a smooth variety embedded in $Y$ as in (\ref{embedding}). Hence the Cartier divisors $D_{1}|_{X_{t_0}},\ldots ,D_{\rho}|_{X_{t_0}}$ form an $\R$-basis for $N^1(X_{t_0})_{\R}$, and the family $\phi :\mathcal{X}\rightarrow T$ satisfies the conditions in Proposition \ref{linear independent}. We conclude that  the map
\[
\rho_{X_t} \ : \ V_{\R}:=\R D_1\oplus\ldots\oplus\R D_{\rho} \ \rightarrow \ N^1(X_t)_{\R}, \ \text{ where } \rho_{X_t}(D_i):=D_i|_{X_t}
\]
is an injective integral linear morphism for all $t\in T$.

With these preparations behind us we can move on to complete the proof. Let us write 
\[
A_t \deq \rho^{-1}_{X_t}(\Neft ), \textup{ and } f_t \deq \textup{vol}_{X_t} \circ \rho_{X_t} 
\]
for each $t\in T$. We need to show that both sets  $(A_t)_{t\in T}$ and $(f_t)_{t\in T}$, are countable. Actually, it is enough to check that there exists a subset $F=\cup F_m \subseteq T$ ($B = \cup B_m \subseteq T$) consisting of a countable union of proper Zariski-closed subsets $F_m \varsubsetneqq T$ (resp. $B_m \varsubsetneqq T$), such that $A_t$ (resp. $f_t$) is independent of $t\in T\setminus F$ ($t\in T\setminus B$). This reduction immediately  implies Theorem~\ref{Theorem}, because one can argue inductively on dim$(T)$ and apply this reduction for each family $\phi :\phi^{-1}(F_m)\rightarrow F_m$ containing a good fiber.

We first prove the above reduction  for  nef cones. The set of all cones $(A_t)_{t\in T}$ has the following property: if $t_o\in T$, then there exists a subset $\cup F_{t_0}^m\varsubsetneqq T$, which does not contain $t_0$ and consists of a countable union of proper Zariski-closed sets such that
\begin{equation}\label{nefcone}
A_{t_0} \ \subseteq \ A_t, \text{ for all } t\in T\setminus \cup F_{t_0}^m\ .
\end{equation}
To verify this claim  choose an element $D\in A_{t_0} \cap \Z^{\rho}$. By \cite[Theorem 1.2.17]{Laz04} on the behaviour of nefness in families, there exists a countable union $F_{t_0,D}\subseteq T$ of proper subvarieties of $T$, not containing $t_0$ such that $D \in A_t$, for all $t$-s outside of $F_{t_0,D}$. As $A_{t_0}$ is a closed pointed cone, the set $A_{t_0} \cap \mathbb{Z}^{\rho}$ is countable and generates $A_{t_0}$ as a closed convex cone. Thus the cone $A_{t_0}$ is included in $A_{t}$ for all $t$'s  outside of the subsets $F_{t_0,D}$ with $D\in A_{t_0} \cap \Z^{\rho}$. Our base field is uncountable, therefore the union of all of the $F_{t_0,D}$'s still remains a proper subset of $T$.

Denoting  $A \deq \cup_{t\in T} A_t$,  it is enough to find a closed point $t\in T$ with $A_t=A$. Note that $A\subseteq V_{\R}=\R^{\rho}$ is second countable, so  there exists a countable set 
\[
\{t_i\in T| \ i\in \N\} \text{ such that } A=\cup_{i\in\N}A_{t_i}\  
\]
according to \cite[Theorem 30.3]{Mun00}. By (\ref{nefcone}), for every $i\in\N$ there exists a countable union of proper Zariski-closed subsets $F_i\varsubsetneqq T$ with the property that 
\[
A_{t_i} \ \subseteq \ A_t, \text{ for all } t\in T\setminus F_i\ ,
\]
and as before $\cup F_i$ remains a proper  subset. This proves~Theorem \ref{Theorem} in the case of nef cones because  we have $A_{t_i}\subseteq A_t$ and hence $A_t=A$ for each $t\in T\setminus \cup F_i$ and $i\in\N$.

Next, we turn out attention to the case of volume functions. We assumed each fiber $X_t$ to be smooth and irreducible. Since the volume function is continuous, and homogeneous of degree $n$, it is actually enough to prove that for any $D\in V_{\Z}$ the volume $\textup{vol}_{X_t}(D|_{X_t})$ is independent of $t\in T\setminus B$. 

Pick a Cartier divisor $D\in V_{\Z}$. By the Semicontinuity Theorem \cite[Theorem III.12.8]{Ha77},
for any $d\in\N$ there exists a proper Zariski-closed subset $B_{D,d}\varsubsetneqq T$, such that 
\[
h^0(t, \mathcal{O}_X(dD)) \ = \ \textup{dim}_{\mathbb{C}} H^0(X_t, \mathcal{O}_X(dD)|_{X_t}) \textup{ is independent of } t\in T\setminus B_{D,d}\ .
\]
The definition of the volume implies that $\textup{vol}_{X_t}(D|_{X_t})$ is independent of $t\in T\setminus \cup_{d\in\N} B_{D,d}$ and, because $V_{\Z}$ is  countable, the union of $\cup B_{D,d}$, for all $D\in V_{\Z}$ and $d\in\N$, is  a countable union of proper Zariski-closed subsets properly contained in $T$. 
\end{proof}


\section{An example of a transcendental volume function}

The aim of this  section is to give an example of a four-fold $X$ where  the volume function $\textup{vol}_X$ is given by a transcendental function over an open subset of $N^1(X)_\R$. We utilize a construction of Cutkosky (see \cite{Cu86} or \cite[Chapter 2.3]{Laz04}) which was also used in \cite{BKS} to produce a non-polynomial volume function (see also \cite{ELMNP}).

Let $E$ be a general elliptic curve, i.e. without complex multiplication. Set $Y=E\times E$. \cite[Lemma 1.5.4]{Laz04} gives a full description of all the cones on $Y$. Let $f_1$, $f_2$ be the divisor classes of  the fibers of the projections $Y\to E$,  and $\Delta$ the class of the diagonal. Then 
\[
\textup{Nef}(Y)_{\R} \ = \ \overline{\textup{Eff}(Y)}_{\R} \ = \ \{ x\cdot f_1+y\cdot f_2+z\cdot\Delta \ | \ xy+xz+yz\geq 0, x+y+z\geq 0\}\ .
\]
Setting  $H_1=f_1+f_2+\Delta$, $H_2=-f_1$ and $H_3=-f_2$, we define the vector bundle
\[
 V \ = \ \mathcal{O}_{E\times E}(H_1)\oplus 
 \mathcal{O}_{E\times E}(H_2)\oplus \mathcal{O}_{E\times E}(H_3)\ ;
\]
$\pi: X=\P (V)\to Y$ will be the four-fold of our interest.
 
\begin{proposition}\label{trans}
With  notation as above, there exists a non-empty open set in $\textup{Big}(X)_{\R}$, where the volume is given by a transcendental formula.
\end{proposition}
\begin{proof}
The characterization of line bundles on projective space bundles,  and the fact that the function $\textup{vol}_X$ is 
continuous, and homogeneous on $\textup{Big}(X)_{\R}$, imply that it is enough to handle $\Q$-divisors of the form
\[
M \ = \ \mathcal{O}_{\P (V)}(1) \otimes \pi^*(\mathcal{O}_Y(L'))
\]
with $L'=c_1f_1+c_2f_2+c_3\Delta$  a $\Q$-Cartier divisor on $Y$ with $(c_1,c_2,c_3)\in\Q_+^3$. 
By the projection formula the volume of $A$ is given by 
\begin{equation}\label{vol}
\textup{vol}_X(M) \ = \ 
\lim_{m\rightarrow \infty}\frac{\sum_{a_1+a_2+a_3=m}h^0(mL'+a_1H_1+a_2H_2+a_3H_3)}{m^4/24}\ ,
\end{equation}
where the sum runs over all $a_i\in\N$'s and the limit over sufficiently divisible values of $m$.
 
In general there is no simple formula in terms of the $a_i$'s for the
right hand side. Nevertheless, when the divisor $mL'+a_1H_1+a_2H_2+a_3H_3$ is ample, then 
\[
 \hh{0}{Y}{mL'+a_1H_1+a_2H_2+a_3H_3} \equ \frac12 ((mL'+a_1H_1+a_2H_2+a_3H_3)^2)
\]
according to the Riemann--Roch theorem on the abelian surface $Y$. 

First, we will show that in the limit as $m$ goes to infinity, the contribution of non-ample
divisors to the sum contained in (\ref{vol}) is negligible. This is done in  the following lemma.
\begin{lemma}
There is a quadratic function of $m$, $F(m)$, such that 
\begin{equation}\label{sum}
\sum_{a_1+a_2+a_3=m, a_i\geq 0} h^0(mL'+a_1H_1+a_2H_2+a_3H_3) \dleq F(m) \ ,
\end{equation}
where the sum runs over all values of $a_1,a_2,a_3$ and $m$, for which $mL'+a_1H_1+a_2H_2+a_3H_3$ is not ample.
\end{lemma}
\begin{proof}
Note that in the sum given in (\ref{sum}), we only need to consider those divisors $D$ of the form $D=mL'+aH_1+a_2H_2+a_3H_3$ that are effective. On the other hand on $Y$ all effective divisors are nef and any non-ample effective divisor $D'$ satisfies $D'^2=0$.

With this in hand, we now show  that for any $m$ there are at most $2(m+1)$ possible choices of $(a_1, a_2, a_3)\in\N^3$ with $a_1+a_2+a_3=m$ and $D$ is non-ample and effective. Indeed, since $a_3=m-a_1-a_2$, on fixing $a_1$ the expression $(mL'+a_1H_1+a_2H_2+a_3H_3)^2$ becomes a quadratic expression in $a_2$ whose $a_2^2$ coefficient is $(H_2-H_3)^2=-2$. This non-zero quadratic expression has at most 2 integral solutions, so for any $a_1\in\{0, \ldots , m\}$ there are at most 2 values for the pair $(a_2, a_3)$ with $a_1+a_2+a_3=m$ and $D$ is effective and non-ample.

It remains to find a bound on $h^0(D)$ which depends only on $m$. Fix  an ample divisor $A$ once and for all; we then have that 
\[
h^0(D)\leq h^0(D+A)= (D+A)^2 \equ  2A\cdot D+ A^2 \equ  2(A\cdot D/m) m + (A^2)\ ,
\]
and that the divisor $D/m$ is contained in the compact set
\[ 
S \deq  \{D' \ | \ D'=L'+b_1A_1+b_2A_2+b_3A_3, b_i  \geq 0, b_1+b_2+b_3=1\} \ ,
\]
 so on setting $N \deq  \max_S (A\cdot D')$, we arrive at the conclusion
\[
h^0(D)\dleq \frac{(D+A)^2}{2} \dleq N m +A^2/2\ . 
\]
Define $F(m) \deq  (m+1)(2Nm+A^2)$ and this quadratic function satisfies (\ref{sum}).
\end{proof}

As a consequence, we can  write our volume function as 
\[
\text{vol}_X(M)=\lim_{m\to\infty}\frac{4!}{2m^4}\cdot \hspace{-1cm}\sum_{\stackrel{a_1+a_2+a_3=m}
{mL'+a_1H_1+a_2H_2+a_3H_3 \mbox{ ample }}} \hspace{-1cm}((mc_1+a_1-a_2)f_1+(mc_2+a_1-a_3)f_2+(mc_3+a_1)\Delta ))^2\ .
\]
Via the substitutions $x= a_2/m$ and $y=a_3/m$ this limit is equal to the integral 
\[
\text{vol}_X(M) \equ  12\int_\Gamma ((1+c_1-2x-y)f_1+(1+c_2-x-2y)f_2+(1+c_3-x-y)\Delta ))^2 
\]
where $\Gamma$ is the subset of $\mathbb{R}^2$ defined by: $x, y \geq 0, x+y\leq 1$ and the class
\[ 
L(x,y) := (1+c_1-2x-y)f_1+(1+c_2-x-2y)f_2+(1+c_3-x-y)\Delta )\mbox{ is ample}.
\]
Setting $q(x,y)= (L(x,y))^2$, we have that
 \[\text{vol}_X(M)= 12\int_\Gamma q(x,y)dxdy.\]
It's not hard to see and also useful to write $q(x,y)=10y^2+B(x)y+C(x)$, where $B(x),C(x)\in \Q(c_1c_2c_3)[x]$. Assume that $c_1,c_2,c_3\in\Q_+$ and $c_1+c_2<1\leq c_1+c_2+2c_3$. This implies that the class $L(0,0)$ is ample and $L(x,y)^2<0$ for any $x,y\geq 0$ with $x+y=1$.

Under these circumstances, $\Gamma$ is the region bounded by:
\begin{enumerate}
\item the x-axis
\item the y-axis and
\item the graph $y= F(x)$, where $F(x)=\frac{B(x)-\sqrt{B^2(x)-40C(x)}}{20}$ is the solution of the equation $q(x,y)=0$. 
\end{enumerate}
Let $X$ be the smallest positive number such that $q(X,0)=C(X)=0$. (Note that
$X\in \overline{\Q(c_1,c_2,c_3)}$.) We can then rewrite our
calculation as
\[ 
\text{vol}_X(M) \equ 12\int_0^X\int_0^{F(x)} q(x,y)dydx
\]
or, in other words,
\[ 
\text{vol}_X(M) \equ  12\int_0^X 10 F(x)^3/3+ B(x) F(x)^2/2+ C(x)F(x)dx\ .
\]
After Euclidean division by the relation $10 F(x)^2+ B(x)F(x)+C(x)=0$ we get
\[ 
\text{vol}_X(M) \equ 12\int_0^X  \frac{40C(x)-(B(x))^2}{60} F(x)dx -12\int_0^X\frac{C(x)B(x)}{60} dx\ .
\]
Denote the second term of the right hand side by $G_1(c_1,c_2,c_3)$. Note that $G_1(c_1,c_2,c_3)\in \overline{\Q(c_1,c_2,c_3)}$ and using the explicit description of $F(x)$, we obtain
\[ 
\text{vol}_X(M) \equ 12\int_0^X  \frac{\big(40C(x)-(B(x))^2\big) \big(B(x)-\sqrt{B^2(x)-40C(x)}\big)}{1200}dx- 
G_1 \ ,
\]
which gives us
\[ 
\text{vol}_X(M) \equ 12\int_0^X  -\frac{\big((B(x))^2-40C(x)\big)^{3/2}}{1200}dx+ G_2(c_1,c_2,c_3)
\]
with 
\[
G_2 \deq -G_1+12\int_0^X \frac{(40C(x)-(B(x))^2) (B(x))}{1200}dx\in \overline{\Q(c_1,c_2,c_3)}\ . 
\]
Let the quadratic function $p(x)=(B(x))^2-40C(x)$ be written in the form $p(x)=ax^2+bx+c$, where  $a,b,c\in \overline{\Q(c_1,c_2,c_3)}$.
We then have that
\[ 
\text{vol}_X(M) \equ -\frac{1}{100}\int_0^X p(x)^{3/2}dx+ G_2(c_1,c_2,c_3)\ ,
\]
and a Maple calculation shows
\[ 
\text{vol}_X(M)= -\frac{1}{100}\Big(\frac{3 (b^2-4ac)^2}{128a^{5/2}}\Big)
\ln{\Big(\frac{b+2aX+ 2(a^2X^2+baX+ca)^{1/2}}{b+2(ca)^{1/2}}\Big)}+G_3(c_1,c_2,c_3)\ ,
\]
where $G_3\in \overline{\Q(c_1,c_2,c_3)}$. 

It remains to check that the function inside the logarithm is not identically $1$ and the one appearing as the coefficient is not identically $0$.  So, take $c_1=c_2=c_3=1/4$. Then 
\[ 
q(x,y) \equ 10y^2+(22x-20)y+75/8-20x+10x^2
\]
which results in $B=22x-20$ and $C=10x^2-20x+75/8$. Furthermore we have that $X= 3/4$ and $B(x)^2-40C(x)= 84x^2-80x+25$, i.e. $a=84$, $b=-80$ and $c=25$. With this in hand, the volume turns our to be
\[ 
\text{vol}_X(M)= -\frac{1}{100}\Big(\frac{15626\sqrt{84}}{98784}\Big)
\ln{\Big(\frac{23+\sqrt{1029}}{-40+\sqrt{2100}}\Big)}+G_3(\frac{1}{4},\frac{1}{4},\frac{1}{4})\ ,
\]
Thus transcendental and this completes the proof of Proposition \ref{trans}.\end{proof}

The transcendental nature of the volume function in a geometrically simple situation leads to a new relation linking complex geometry to diophantine questions. 
In their inspiring work \cite{KZ} (see also \cite{W}), Kontsevich and Zagier write about the ubiquitous nature of periods. According to their
definition, a  complex number $\alpha$ is a period, if it can be written as the integral of a rational  function with rational coefficients over an algebraic
domain (a subset or Euclidean space determined by polynomial inequalities with rational coefficients).

By their very definition, periods are countable in number, and contain all algebraic numbers. On the other hand, various transcendental numbers  manifestly belong to this circle, $\pi$ or the natural logarithms of positive integers among them. It's easy to verify that periods form a ring with respect to the usual operations on real numbers. Although it is obvious from cardinality considerations that most real numbers are not elements of this ring, so far only one real number has been proven \emph{not} to be a period by Yoshinaga \cite{Yo08}. 

Using results of Lazarsfeld and Musta\c t\u a from \cite{LM08}, the volume of a Cartier divisor $D$ can be written as 
\[
 \vol_X(D) \equ \int_{\Delta_{Y_\bull}{(D)}} 1\ ,
\]
where $\Delta_{Y_{\bull}}(D)$ is the Okounkov body of $D$ with respect to any admissible complete flag $Y_\bull$ of subvarieties in $X$. This means that
$\vol(D)$ --- originally defined as the asymptotic rate of growth of the number of global sections of multiples of $D$ --- is a period in a very  natural way,
whenever $\Delta_{Y_{\bull}}(D)$ is an algebraic domain for some suitably chosen admissible flag. This happens in all the cases that have been explicitly computed so far, leading to the following question.
\begin{qsn}
Is  the volume of an integral Cartier divisor on an irreducible projective variety always a period?
\end{qsn}


\begin{thebibliography}{RWY}
\bibitem{Al39}
A. D. Alexandroff, \emph{Almost everywhere existence of the second differential of a convex function and some properties of convex surface connected with it}, Leningrad State Univ. Ann., Mat., Ser. {\bf 6} (1939), pp.3-35 (Russian).

\bibitem{BKS} Th. Bauer, A. K\"uronya, T. Szemberg: \emph{Zariski chambers, volumes, and stable base loci}, Journal f\"ur die reine und angewandte Mathematik \textbf{576} (2004), 209--233. 

\bibitem{BFJ07}
S. Boucksom, C. Favre, M. Jonsson, \emph{Differentiability of volumes of divisors and a problem of Teissier},
arxiv.org.

\bibitem{CP90}
F. Campana and T. Peternell, \emph{Algebraicity of the ample cone of projective varieties},
J. Reine Angew. Math. \textbf{407} 1990, pp. 160-166.

\bibitem{Cu86}
S. Dale Cutkosky, \emph{Zariski decomposition of divisors on algebraic varieties}, 
Duke Math. J. \textbf{53} (1986), no. 1, 149-156.

\bibitem{ELMNP} L. Ein, R. Lazarsfeld, M. Musta\c t\u a, M. Nakamaye, M. Popa: \emph{Asymptotic invariants of base loci}, Ann. Inst. Fourier (Grenoble)
\textbf{56} (2006), No. 6., 1701--1734.


\bibitem{Fu94} 
T. Fujita, \emph{Approximating Zariski decomposition of big line bundles}, 
Kodai Math. J. \textbf{17} (1994), no. 1, 1-3. 


\bibitem{Gr}
A. Grothendieck, \emph{El\'ements de g\'eom\'etrie alg\'ebrique} IV 3, Publ.Math. IHES,
\textbf{28}, (1966).

\bibitem{HM06}
C. Hacon, J. McKernan, \emph{Boundedness of pluricanonical maps of varieties of general type.}
Invent. Math. \emph{166} (2006), no. 1, 1-25.

\bibitem{HS04}
M. Haiman and B. Sturmfels, \emph{Multigraded Hilbert schemes},
J. Alg. Geom. \textbf{13}, no. 4, 725-769, 2004.

\bibitem{Ha77}
R. Hartshorne, \emph{Algebraic Geometry}, 
Graduate Texts in Mathematics, vol. \textbf{52}, Springer-Verlag, New York, 1977.

\bibitem{HKP} M. Hering, A. K\"uronya, S. Payne: \emph{Asymptotic cohomological functions of toric divisors}, Advances in Mathematics \textbf{207}
No. 2. (2006), 634--645.

\bibitem{KL66}
S. Kleiman, \emph{Towards a numerical theory of ampleness},
Ann. of Math., \textbf{84}, 293-344, 1966.

\bibitem{KZ} M. Kontsevich, D. Zagier: \emph{Periods}, in  B. Enquist, W. Schmied (eds.): Mathematics unlimited -- 2001 and beyond. Berlin, New York, Springer Verlag, 771--898.

\bibitem{Laz04}
R. Lazarsfeld, \emph{Positivity in Algebraic Geometry I-II},
Ergebnisse der Mathematik und ihrer Grenzgebiete. 3. Folge, vol. \textbf{48-49}, Springer-Verlag, Berlin
Heidelberg, 2004.

\bibitem{LM08}
R. Lazarsfeld and M. Musta\c t\u a, \emph{Convex bodies associated to linear series},
Ann. Scient. \' Ec. Norm. Sup., {\bf 4} s\' erie, t. {\bf 42}, (2009), pp 783-835.

\bibitem{Mun00}
J.R. Munkres, \emph{Topology}, Prentice Hall, Upper Saddle River, New Jersey, 2000.

\bibitem{Na04}
N. Nakayama, \emph{Zariski-decomposition and abundance}, 
MSJ Memoirs, vol. 14, Mathematical Society of Japan, Tokyo, 2004.

\bibitem{Ok96} A. Okounkov, \emph{Brunn-Minkowski inequalities for multiplicities}, Invent. Math {\bf 125} 
(1996) pp 405-411. 

\bibitem{Ok03} A. Okounkov, \emph{Why would multiplicities be 
log-concave?} in The orbit method in geometry and physics, Progr. Math. {\bf 213}, 2003 pp 329-347.

\bibitem{Sch93} R. Schneider, \emph{Convex Bodies: The Brunn-Minkowski Theory},
 Encyclopedia of Mathematics and its Applications, vol. {\bf 44}. Cambridge University Press, Cambridge, 1993

\bibitem{Sha94}
I. R. Shafarevich, \emph{Basic Algebraic Geometry: Varieties in Projective Space}, Springer, Berlin Heidelberg New York, 1994.

\bibitem{Ta06}
S. Takayama, \emph{Pluricanonical systems on algebraic varieties of general type.},
Invent. Math. \textbf{165} (2006), no. 3, 551-587.

\bibitem{Ts99}
H. Tsuji, \emph{On the structure of the pluricanonical systems of projective varieties of general type},
preprint, 1999.

\bibitem{Yo08} M. Yoshinaga: \emph{Periods and elementary real numbers}, preprint, \href{http://arxiv.org/abs/0805.0349v1}{arXiv:0805.0349v1}.

\bibitem{W}
M. Waldschmidt, \emph{Transcendence of periods: the state of the art}, Pure and Applied Mathematics Quarterly \textbf{2} (2), 435--463. 

\end{thebibliography}
\end{document}